\documentclass[10pt,a4paper,leqno]{amsart}
\usepackage{etex}
\usepackage{amsfonts}
\usepackage{mathrsfs}

\usepackage{comment}

\usepackage{amsthm}
\usepackage{xspace}
\usepackage{graphics}
\usepackage[english]{babel}
\usepackage{amsfonts,amssymb}
\usepackage[all]{xy}
\usepackage{stmaryrd}
\usepackage{color}

\usepackage{geometry}                % See geometry.pdf to learn the layout options. There are lots.
\usepackage{amssymb}
\usepackage{tikz}
\usetikzlibrary{calc}
\usetikzlibrary{decorations.pathreplacing,calc,decorations.pathmorphing}
\usepackage{tikz-qtree}
\usetikzlibrary{fit,shapes,calc,quotes,angles,arrows.meta,intersections} %Loading the calc library to compute the vector of two points, then you can draw perpendicular lines. 
\usepackage{tkz-euclide}
\usetkzobj{all}
\newcommand{\gettikzxy}[3]{%
  \tikz@scan@one@point\pgfutil@firstofone#1\relax
  \edef#2{\the\pgf@x}%
  \edef#3{\the\pgf@y}%
  }

\usepackage[colorlinks,linkcolor=blue,citecolor=blue,urlcolor=blue]{hyperref}

\newcommand{\tikzAngleOfLine}{\tikz@AngleOfLine}
  \def\tikz@AngleOfLine(#1)(#2)#3{%
  \pgfmathanglebetweenpoints{%
    \pgfpointanchor{#1}{center}}{%
    \pgfpointanchor{#2}{center}}
  \pgfmathsetmacro{#3}{\pgfmathresult}%
  }

\newcommand{\be}{\begin{otherlanguage}{english}}
\newcommand{\ee}{\end{otherlanguage}}

\theoremstyle{definition}
\newtheorem{defn}{Definition}
\newtheorem{que}{Question}

\theoremstyle{plain}
\newtheorem{lem}{Lemma}
\newtheorem{prop}{Proposition}
\newtheorem{thm}{Theorem}

\newtheorem*{thm*}{Theorem}

\theoremstyle{remark}

\numberwithin{equation}{subsection}

\newcommand{\beq}{\begin{equation}}
\newcommand{\eeq}{\end{equation}}

\newcommand{\N}{\mathbb{N}}
\newcommand{\Z}{\mathbb{Z}}

\newcommand{\Q}{\mathbb{Q}}

\def\Id{{\rm Id}}

\newcommand{\norm}[1]{\| #1\|}

\def\a{          \alpha}

\def \C{{\mathbb C}}
\def \R{{\mathbb R}}

\def \Z{{\mathbb Z}}
\def \N{{\mathbb N}}

\newcommand{\T}{{\mathbb T}}
\newcommand{\prf}{{\begin{proof}}}
\newcommand{\epf}{{\end{proof}}}

\newcommand{\ary}{\begin{eqnarray}}
\newcommand{\eary}{\end{eqnarray}}

\newcommand{\aryst}{\begin{eqnarray*}}
\newcommand{\earyst}{\end{eqnarray*}}

\newcommand{\enmt}{\begin{enumerate}}
\newcommand{\eenmt}{\end{enumerate}}

\theoremstyle{definition}

\def\bee{\begin{equation}}
\def\eee{\end{equation}}

\theoremstyle{rema}

\newtheorem{rema}{\sc Remark}

\newcommand{\Diff}{\text{Diff}}

\numberwithin{equation}{section}

\begin{document}
\title{A QUESTION OF NORTON-SULLIVAN in the analytic case}

%----------Author 1
\author{Jian Wang}
\address{IMPA, Estrada Dona Castorina, 110, 22460-320, Rio de Janeiro, Brazil}\email{jian.wang@impa.br}
\thanks{J. Wang acknowledges the support of CNPq- Brazil (Bolsa de P\'os-Doutorado J\'unior) and H. Yang  acknowledges the support of FAPERJ- Brazil (Programa Mestrado Nota 10). }
%----------Author 2

\author{Hui Yang}
\address{IMPA, Estrada Dona Castorina, 110, 22460-320, Rio de Janeiro, Brazil}
\email{rengyanghui@gmail.com}
%\thanks{}

\date{Sep. 17, 2018}
\maketitle
\begin{abstract}
In 1996, A. Norton and  D. Sullivan  asked the following question:
If $f:\mathbb{T}^2\rightarrow\mathbb{T}^2$ is a diffeomorphism, $h:\T^2\rightarrow\T^2$ is a continuous map homotopic to the identity, and $h f=T_{\rho} h$ where $\rho\in\R^2$ is a totally irrational vector and $T_{\rho}:\mathbb{T}^2\rightarrow\mathbb{T}^2,\, z\mapsto z+\rho$ is a translation, are there natural geometric conditions (e.g. smoothness) on $f$ that force $h$ to be a homeomorphism?
In \cite{WZ}, the first author and  Z. Zhang gave a negative answer to the above question in the $C^{\infty}$ category:  In general, not even the infinite smoothness condition can force $h$ to be a homeomorphism. In this article, we give a negative answer in the $C^{\omega}$ category: We construct a real-analytic conservative and minimal totally irrational pseudo-rotation of $\T^2$ that is semi-conjugate to a translation but not conjugate to a translation, which simultaneously answers a question raised in \cite[Q3]{WZ}.  
\end{abstract}

\section{Introduction}

As one of the earliest results, H. Poincar\'{e} proved the following celebrated classification of circle homeomorphisms:  a circle homeomorphism $f$ is semi-conjugate to an irrational  rigid rotation if and only if  the rotation number of $f$, denoted by $\rho(f)$, is irrational, which is equivalent to say that $f$ has no periodic orbits. 
Later, A. Denjoy proved that $f$ is topologically conjugate to an irrational rigid rotation if it is a $C^1$ diffeomorphism of $\T^1$ without periodic points and $Df$ has bounded variation ($f\in C^{1+b.v.}$) \cite{D}.  In the other direction, Denjoy (even before him, P. Bohl \cite{Bo}) provided examples of $C^1$ diffeomorphisms semi-conjugate but not topologically conjugate to an irrational rotation. Their examples were later improved to $C^{1+\alpha}$ for any $\alpha\in(0,1)$  by Herman  \cite{H}.  

It is natural to explore the Denjoy's results (Denjoy Theorem and Denjoy counter-examples) on higher dimensional tori. It is the motivation for a line of research on the extension of the Denjoy's type example of the circle to $\T^2$. To construct a Denjoy counter-example on the circle, one starts with an irrational rotation and blows up the orbit of some point to get an orbit of wandering intervals. Inspired by this, one motivating question is the \textit{wandering domains problem} (see \cite{NS}): 
Can one ``\,blow up\,'' one or more orbits of $T_{\alpha}$ to make a smooth diffeomorphism with wandering domains?  
We say that a homeomorphism of $\T^2$ is of Denjoy type if it is obtained by blowing-up finitely many orbits of an irrational translation.  P. McSwiggen in \cite{McS} constructed a $C^{2+\alpha}$ diffeomorphism of Denjoy type having a smooth wandering domain. In particular, his example is not topologically conjugate to a rigid translation.
Norton and Sullivan in \cite{NS} showed that there dose not exist  $C^3$ diffeomorphism on $\T^2$ of Denjoy type with circular wandering domains, and asked the following question:
\begin{que}[Norton and Sullivan, 1996]\label{qnortonsullivan} 
If $f:\mathbb{T}^2\rightarrow\mathbb{T}^2$ is a diffeomorphism, $h:\T^2\rightarrow\T^2$ is a continuous map homotopic to the identity, and $h f=T_{\rho} h$ where $\rho\in\R^2$ is a totally irrational vector, are there natural geometric conditions (e.g. smoothness) on $f$ that force $h$ to be a homeomorphism?
\end{que}
\noindent In \cite{PS}, A. Passeggi and M. Sambarino  also mentioned the question: whether there exists $r$ so that if $f : \T^2 \to \T^2$ is a $C^r$ diffeomorphism semi-conjugate to an ergodic translation, then $f$ is conjugate to it.  For more recent developments, we mention \cite{Kara, Nav, M, WZ}.

In \cite{WZ}, the first author and  Zhang constructed a smooth diffeomorphism which is isotopic to the identity and semi-conjugate to a minimal translation $T_\a$, but not conjugate to $T_\a$, which is a $C^\infty$ counter-example to the Norton-Sullivan's question. The construction in \cite{WZ} combined the classical Anosov-Katok method (see \cite{AK,FK}) with J\"ager's theorem \cite{Jage} (see Theorem \ref{thm:J} below). In this article, we will construct  a $C^{\omega}$ counter-example to the question of Norton-Sullivan. Our strategy in this paper mainly follows from the \textit{approximation by conjugation construction scheme} in the proof of  Theorem 4 in \cite{WZ}. However, as we require the map is real-analytic, we will apply certain technique of analytic approximations in \cite{B,BSK} to customize the desirable analytic conjugacies. Our main theorem is the following:

\begin{thm}\label{thm example}
For any integer $d \geq 2$, there exists a $C^{\omega}$ area-preserving and minimal map $f : \T^d \to \T^d$ which is semi-conjugate to a minimal translation by a map homotopic to the identity, but is not topologically conjugate to a translation. \end{thm}

The classical Anosov-Katok method is a major source constructing examples of smooth dynamical systems with prescribed properties. This method is well known and was applied by many authors to construct different examples which satisfy some desired properties, e.g. ergodic, mixing, minimal, etc (see, e.g. \cite{AK,S,FK,FKs,BSK}). We would not want to restate this scheme in our article and instead, we recommend the classical articles \cite{AK,FK}. The conjugation by approximation construction (i.e. the Anosov-Katok method) is essentially nonlinear and it is based on the convergence of maps obtained from certain standard maps by wildly diverging conjugacies. There is a great difference between the differentiable and real-analytic maps becomes apparent (see \cite[Section 7.2]{FK} for the explanation). Hence, one will meet additional difficulties when one considers to construct examples of real-analytic diffeomorphisms by using this method. One possible way to overcome such difficulties is to work on some manifolds which have a large collection of real-analytic diffeomorphisms with some good properties, and whose singularities are uniformly bounded away from a complex neighborhood of the real domain (see, e.g. B. Fayad and A.B. Katok \cite{FKs} work on odd-dimensional spheres).  In our situation, we will work on the torus and use a trick on the approximation by conjugation scheme appeared in \cite{B,BSK} recently, which trick that can be traced earlier to Katok \cite{K}.

We give some remarks about  our theorem. 
As the constructions in \cite{WZ} and in this article, we use the classical Anosov-Katok method, the rotation vector of the constructed map is Liouvillean, which is the price to pay in order to get the smoothness of the pseudo-rotation. It seems difficult to construct a pseudo-rotation as in Theorem \ref{thm example} with Diophantine rotation vector (see Section 2 for the definitions). On the other hand, by the classical KAM theory,  any $C^{\infty}$ volume-preserving pseudo-rotation of $\T^n$ with Diophantine rotation vector $\alpha \in \T^n$, which is sufficiently close to $T_{\alpha}$, is smoothly conjugate to $T_{\alpha}$. Hence, for the Norton-Sullivan question, except the smoothness condition, the arithmetic condition of the rotation vector of the pseudo-rotation is also vital. Therefore, we ask the following question:
\begin{que}
In Question \ref{qnortonsullivan}, if the vector $\rho$ is diophantine, is the Norton-Sullivan's question true?
\end{que}

\smallskip
This article is organized as follows. In Section 2, we introduce some notations, recall some classical definitions and results. In particular, we introduce the block-slide type of maps and their analytic approximations. In Section 3, we customize the analytic conjugacies  which is a key step to prove our main theorem. We prove the main theorem in Section 4. \smallskip
\bigskip

\section{Prelimary}
\subsection{The Misiurewicz-Ziemian rotation set} In this article, we  study homeomorphisms of the two-dimensional torus $\T^2 = \R^2 / \Z^2$ which are isotopic to the identity. In this case, the rotation vectors and the rotation set are defined as follows.

Let $\mathrm{Homeo}_*(\mathbb{T}^2)$ be the group of homeomorphisms of $\mathbb{T}^2$ which are homotopic to $\mathrm{Id}_{\mathbb{T}^2}$.\footnote{If a homeomorphism of $\T^2$ is homotopic to the identity, then it is  isotopic to the identity \cite[Theorem 6.4]{Eps}.}  Any $f \in \mathrm{Homeo}_*(\mathbb{T}^2)$ admits a lift to $\R^2$, denoted by $\tilde{f}$, which is a homeomorphism of $\R^2$ satisfying $\pi\tilde{f} = f\pi$, where  $\pi:\R^2\rightarrow \T^2$ is the covering projection. 

M. Misiurewicz and K. Ziemian \cite{MZ} introduced the following standard definition:
\begin{defn}Assume that $f\in \mathrm{Homeo}_*(\mathbb{T}^2)$ and that $\tilde{f}$ is a lift of $f$. The \textit{(Misiurewicz-Ziemian) rotation set} of $\tilde{f}$ is defined by:
\aryst
\rho(\tilde{f}) = \left\{ {\bf v} \in \R^2  \mid \frac{\tilde{f}^{n_i}(z_i)-z_i}{n_i} \to {\bf v},\mbox{for some } \{z_i\} \mbox{ in }
\mathbb{R}^2,\mbox{ and } \{n_i\} \mbox{ in }\N \mbox{ with }n_i \rightarrow \infty\right \}.
\earyst
\end{defn} 
The effect of changing the lift $\tilde{f}$ of $f$  is to translate $\rho(\tilde{f})$ by an integer vector.
In \cite{MZ}, the authors proved  that the rotation set $\rho(\tilde{f})$ is a compact convex subset of $\R^2$, giving rise to a basic trichotomy: $\rho(\tilde{f})$ is either a compact convex set with nonempty interior, a line segment, or a singleton. 
We say
that $f$ is a \emph{pseudo-rotation}  when  $\rho(\tilde{f})$ is a singleton. Moreover, we say a pseudo-rotation $f$ is \emph{totally irrational} if $\rho(\tilde{f})=(\rho(\tilde{f})_1,\rho(\tilde{f})_2)$ satisfies that $\rho(\tilde{f})_1,\rho(\tilde{f})_2\notin \Q$ and they are \textit{non-resonant (or rational independent)}, that is, for any $(a,b,c)\in\Z^3$ satisfying $a\rho(\tilde{f})_1+b\rho(\tilde{f})_2+c=0$ implies that $(a,b,c)=(0,0,0)$. 

For $\gamma, \sigma>0$, we define the set $\mathcal{D}(\gamma,\sigma)\subset \R^2$ of \textit{diophantine} vector with \textit{exponent} $\sigma$ and \textit{constant} $\gamma$ as the set of $\alpha=(\alpha_1,\alpha_2)\in\R^2$ such that 
$$\forall (k_1,k_2)\in\Z^2,\, |k_1\alpha_1+k_2\alpha_2|\geq\frac{\gamma}{(|k_1|+|k_2|)^\sigma}.$$ 
We set $\mathcal{D}(\sigma)=\bigcup_{\gamma>0}\mathcal{D}(\gamma,\sigma)$ and $\mathcal{D}=\bigcup_{\sigma>0}\mathcal{D}(\sigma)$. The set $\mathcal{D}$ is the set of \textit{Diophantine} vectors of $\R^2$ while its complement in the set of non-resonant vectors is called the set of \textit{Liouville} vectors, denoted it by $\mathcal{L}$. We note that the set $\mathcal{D}$ has full Lebesgue measure in $\R^2$ and the set $\mathcal{L}$ is $G_\delta$-dense in $\R^2$. In the same way, one can define all of the definitions above in higher dimensions.

\subsection{Semi-conjugation}
We denote by  $\|\cdot\|$  the Euclidean norm on $\mathbb{R}^2$ and by $d$ the standard  Euclidean metric. Let $\T^2$ be endowed with the  metric induced by the Euclidean metric on $\mathbb{R}^2$, we still denote it by $d$ without any confusion.

For $\a=(\a_1,\a_2)\in \R^2$, we define the transition $T_\a:\R^2 \rightarrow \R^2$ by
$T_\a(x,y)=(x+\a_1,y+\a_2)$ which naturally induces a translation on $\T^2$, we still denote it by $T_\a$ without any confusion.
Given a map  $f: \T^2 \rightarrow \T^2$, we say that $f$ is a \textit{semi-conjuagte} to a translation $T_\alpha$ if  there exists a surjective continuous map $h: \T^2 \rightarrow \T^2$, such that $hf=T_\a h$, moreover, if $h$ is a homeomorphism, we say that $f$ is \textit{conjugate} to a translation.

\begin{defn}\label{BDC} Let $f$ be a pseudo-rotation of $\T^2$. We say that  $f$ has \emph{bounded mean
motion } (with a bound $\kappa\geq0$) if there exists $\tilde{f}$, a lift of $f$,  such that   for any $ z\in
\mathbb{R}^2$ and $ n\in \mathbb{N}$,
\begin{equation}\label{eq:bmm}\|\tilde{f}^n(z)-z-n\rho(\tilde{f})\|\leq \kappa. \end{equation}  
\end{defn}

%Let $\mathrm{Homeo}_*(\mathbb{T}^2)$ be the group of homeomorphisms of $\mathbb{T}^2$ which are isotopic to $\mathrm{Id}_{\mathbb{T}^2}$. 

\begin{thm}\cite[proposition A]{Jage}\label{thm:J}
Suppose that $f \in \mathrm{Homeo}_*(\T^2)$ is a conservative minimal totally irrational pseudo-rotation with bounded mean motion. Then $f$ is  semi-conjugate to an irrational translation on $\T^2$ and the semi-conjugacy is homotopic to the identity.
\end{thm}
% So it will be enough to construct a minimal totally irrational pseudo-rotation with bounded mean motion to get the semi-conjugate function we want.

\subsection{Analytic topology}
Any real-analytic diffeomorphim $f$ of $\T^2$ homotopic to the identity admits a lift $F: \R^2 \rightarrow \R^2$ which has the following form:
$$F(x_1,x_2)=(x_1+f_1(x_1,x_2), x_2+f_2(x_1+x_2)),$$
where $f_1,f_2$ are real analytic $\Z^2$ periodic functions. It can be extend to some neighborhood of $\R^2$ in $\mathbb{C}^2$. For any $\rho>0$, let
 $$B_\rho=\{(z_1,z_2)\in \mathbb{C}^2\mid ~|\mathrm{Im}(z_1)|,|\mathrm{Im}(z_2)|<\rho\},$$ and for a function $h$ defined on $B_\rho$, we  define  $$\norm{h}_\rho=\sup\limits_{(z_1,z_2)\in B_\rho} |h(z_1,z_2)|.$$
We define $C^\omega_\rho(\T^2)$ to be the space of all $\Z^2$-periodic real-analytic function on $\R^2$ that extends to a holomorphic function on $B_\rho$ and $\|h\|_\rho<\infty$.

Let $\lambda$ be the standard Lebesgue measure on $\T^2$. We denote by $\Diff_\rho^{\,\omega}(\T^2,\lambda)$ the space of all measure-preserving real-analytic diffeomorphism  of $\T^2$  homotopic to the identity, whose lift $F(x)=(x_1+f_1(x),x_2+f_2(x))$ to $\R^2$ satisfies $f_i\in C^\omega_\rho$ and we also require the lift $\tilde{F}(x)=(x_1+\tilde{f}_1(x),x_2+\tilde{f}_2(x))$ of its inverse to $\R^2$ to satisfy $\tilde{f}_i\in C^\omega_\rho$. 
%Let $F=(F_1,F_2)$ be the lift of diffeomorphism in $\Diff_\rho^{\,\omega}(\T^2,\lambda)$. We define the norm of the total derivative $$\|DF\|_\rho:=\max_{i,j\in\{1,2\}}\left\|\frac{\partial F_i }{\partial x_j}\right\|_\rho.$$
For any $f,g\in\Diff_\rho^{\,\omega}(\T^2,\lambda)$, we define the distance
$$d_\rho(f,g)=\max\{\tilde{d}_\rho(f,g),\tilde{d}_\rho(f^{-1},g^{-1})\},$$
where
$$\tilde{d}_\rho(f,g)=\max\limits_{i=1,2} \left\{ \inf\limits_{k\in \Z}||f_i(z_1,z_2)-g_i(z_1,z_2)+k||_\rho\right\}.$$Finally, we define the space: $\Diff_\infty^{\,\omega}(\T^2,\lambda):=\bigcap_{n=1}^\infty \Diff_n^{\,\omega}(\T^2,\lambda)$. For more information about the analytic topology, we recommend the readers to refer to \cite{BSK,S}.

\subsection{Analytic approximations}
 In this subsection, we will introduce two lemmas. Given a special kind of step function with some periodic propriety (it is called \textit{block-slide type of maps} in \cite{B,BSK}, see (\ref{eq:s}) below), we can construct an explicit form of some analytic approximations of the function, which preserves the periodic propriety and satisfies some Lipschitz condition, that is crucial to the proof of the main result of this article.  The key point is that a block-slide type of map can be approximated extremely well by measure-preserving real-analytic diffeomorphisms outside a set of arbitrarily small measure, which is inspired by \cite{K}.

Let $q\in \N, N \in 2\N$,  and $\beta=(\beta_0,\cdots,\beta_{N-1})\in [0,1)^N$. Consider a step function of the form
\ary\label{eq:s}
\bar{s}_{\beta,q}:[0,1)\rightarrow \R \text{ defined by }\,\bar{s}_{\beta,q}=\sum_{j=0}^{Nq-1}\bar{\beta_j}\chi_{[j/Nq,(j+1)/Nq)}.\eary
Here, $\bar{\beta_j}:=\beta_k$, where $ k := j
\,(\mathrm{mod}\,  N)$. For any $\delta\in(0,1)$, we denote by $F_{q,N,\delta}$ the union of all intervals centered around $j/Nq\,(j\in \Z)$ with length $\delta/Nq$. For given $\epsilon\in (0,\frac{1}{8})$ and $\delta\in (0,1)$, we define $$A_0(\epsilon,\delta, N):=\max\left\{-\frac{2N}{\pi\cdot\delta}\cdot\ln(-\ln(1-\frac{\epsilon}{8})),\frac{2N}{\pi\cdot\delta}\cdot\ln(-\ln(\frac{\epsilon}{2N}))\right\}.$$ 
\begin{lem}\label{lem:aa} \cite[Lemmas 2.13, 2.18, 3.14]{BSK} For any $\epsilon\in (0,\frac{1}{8}),\delta\in (0,1)$ and any $A>A_0(\epsilon,\delta,N)$, we define the following $1/q$-periodic real-analytic map $\tilde{s}_{\beta,q,\epsilon,\delta,A}: \R\rightarrow \R $ as
\ary\label{eq:abc}
\tilde{s}_{\beta,q,\epsilon,\delta,A}(x)
&=& (\displaystyle{\sum_{j=0}^{N/2-1}}\beta_j (e^{-e^{-Asin2\pi(qx-j/N)}}-e^{-e^{-Asin2\pi(qx-(j+1)/N)}}))e^{-e^{-Asin2\pi qx}}\\
&+& (\displaystyle{\sum_{j=N/2}^{N-1}} \beta_j (e^{-e^{-Asin2\pi(qx-j/N)}}-e^{-e^{-Asin2\pi(qx-(j+1)/N)}}))e^{-e^{Asin2\pi qx}}\nonumber.
\eary
 The map $\tilde{s}_{\beta,q,\epsilon,\delta,A}$ has the following properties:
%there exist a periodic real analytic function $\tilde{s}_{\beta, q}: \R\rightarrow \R$ satisfying:
\enmt
\item[$(1).$] The complexification of $\tilde{s}_{\beta,q,\epsilon,\delta,A}$ extends holomorphically to $\C$;
\item[$(2).$] We have
$\sup_{x\in [0,1)\setminus F_{q,N,\delta}}|\tilde{s}_{\beta,q,\epsilon,\delta,A}(x)-\bar{s}_{\beta,q}(x)|< \epsilon;$
\item[$(3).$] The map $\tilde{s}_{\beta,q,\epsilon,\delta,A}$ is $\frac{1}{q}$-periodic. More precisely, the complexification of $\tilde{s}_{\beta,q,\epsilon,\delta,A}$ satisfies 
$$\tilde{s}_{\beta,q,\epsilon,\delta,A}(z+k/q)=\tilde{s}_{\beta,q,\epsilon,\delta,A}(z) \text{ for all } z\in \C \text{ and }k\in\Z;$$
\item[$(4).$] $\forall \rho>0$, there exist a constant $C(N,q,\epsilon,\delta,A,\rho)>0$ such that:
$$\sup_{z_1,z_2\in B_\rho}|\tilde{s}_{\beta,q,\epsilon,\delta,A}(z_1)-\tilde{s}_{\beta,q,\epsilon,\delta,A}(z_2)|\leq C|z_1-z_2|.$$
\eenmt

\end{lem}

\begin{rema}\label{rem:bound}Note that $0<e^{-e^t}<1$ for all $t\in \R$. It is obviously that $|\tilde{s}_{\beta,q,\epsilon,\delta,A}(z)| < \displaystyle{\sum_{i=0}^{N-1}} |\beta_i|.$
\end{rema}

\begin{rema}\label{rem:e}
We can take the constant  $C(N,q,\epsilon,\delta,A,\rho)$ in the item $(4)$ as $6\pi\cdot A\cdot N\cdot q\cdot e^{4\cdot e^{A \cdot e^{2\pi q\rho}}}$ (see the proof of Lemma 3.14 in \cite{BSK}). 
\end{rema}

For every $m\geq1$, we define $$A^{(m)}_0(\epsilon,\delta,N):=A_0(\frac{\epsilon}{4m},\delta,N) \text{ and } C^{(m)}(N,q,\epsilon,\delta,A,\rho):=4mC(N,q,\epsilon,\delta,A,\rho).$$ 

In order to prove our main theorem, we require that $\beta\in [-1,1)^N$. Hence, we give the following lemma.

\begin{lem}\label{R}Let $q\in \N, N \in 2\N$, and $\beta=(\beta_0,\cdots,\beta_{N-1})\in [-m,m)^N$ where $m\in \N_{\geq1}$.
Suppose that $\bar{s}_{\beta,q}$ is the step function defined in (\ref{eq:s}). Then for any $\epsilon\in (0,\frac{1}{8}),\delta\in (0,1)$ and any $A>A^{(m)}_0(\epsilon,\delta,N)$, the function $\tilde{s}_{\beta,q,\epsilon,\delta,A}$ defined in (\ref{eq:abc}) satisfies all of the properties of Lemma \ref{lem:aa} if we replace the constant $C$ in the property (4) by $C^{(m)}(N,q,\epsilon,\delta,A,\rho)$.
\end{lem}
\begin{proof}
Let $\beta_0=(\frac{1}{2},\cdots,\frac{1}{2})\in [0,1)^N$. Then $\frac{1}{2m}\beta+\beta_0\in [0,1)^N$. For any $\epsilon\in (0,\frac{1}{8}),\delta\in (0,1)$, applying $\frac{\epsilon}{4m},\delta$ and $A>A_0(\frac{\epsilon}{4m},\delta,N)$ to Lemma \ref{lem:aa}, we have
 $$\sup_{x\in [0,1)\setminus F_{q,N,\delta}}|\tilde{s}_{\frac{1}{2m}\beta+\beta_0,q,\epsilon,\delta,A}(x)-\bar{s}_{\frac{1}{2m}\beta+\beta_0,q}(x)|< \frac{\epsilon}{4m},$$
$$\sup_{x\in [0,1)\setminus F_{q,N,\delta}}|\tilde{s}_{\beta_0,q,\epsilon,\delta,A}(x)-\bar{s}_{\beta_0,q}(x)|< \frac{\epsilon}{4m}.$$
 Note that the maps $\beta \rightarrow \bar{s}_{\beta,q}, \tilde{s}_{\beta,q,\epsilon,\delta,A}$ are  linear. Hence, we get
\aryst
&&\sup_{x\in [0,1)\setminus F_{q,N,\delta}}|\tilde{s}_{\beta,q,\epsilon,\delta,A}(x)-\bar{s}_{\beta, q}(x)|\\&\leq& \sup_{x\in [0,1)\setminus F_{q,N,\delta}}|\tilde{s}_{\beta+2m\beta_0,q,\epsilon,\delta,A}(x)-\bar{s}_{\beta+2m\beta_0,q}(x)|+\sup_{x\in [0,1)\setminus F_{q,N,\delta}}|\tilde{s}_{2m\beta_0,q,\epsilon,\delta,A}(x)-\bar{s}_{2m\beta_0,q}(x)|\\&< &2m\frac{\epsilon}{4m}+2m\frac{\epsilon}{4m}=\epsilon.
\earyst
This is the property (2). Similarly,
%$$\sup_{z_1,z_2\in B_\rho}|\tilde{s}_{\frac{1}{2m}\beta+\beta_0,q,\epsilon,\delta,A}(z_1)-\tilde{s}_{\frac{1}{2m}\beta+\beta_0,q,\epsilon,\delta,A}(z_2)|\leq kC|z_1-z_2|.$$
\aryst
&&\sup_{z_1,z_2\in B_\rho}|\tilde{s}_{\beta,q,\epsilon,\delta,A}(z_1)-\tilde{s}_{\beta,q,\epsilon,\delta,A}(z_2)|\\&\leq&\sup_{z_1,z_2\in B_\rho}|\tilde{s}_{\beta+2m\beta_0,q,\epsilon,\delta,A}(z_1)-\tilde{s}_{\beta+2m\beta_0,q,\epsilon,\delta,A}(z_2)|+\sup_{z_1,z_2\in B_\rho}|\tilde{s}_{2m\beta_0,q,\epsilon,\delta,A}(z_1)-\tilde{s}_{2m\beta_0,q,\epsilon,\delta,A}(z_2)|\\&\leq &2m\left(C(N,q,\epsilon,\delta,A,\rho)+C(N,q,\epsilon,\delta,A,\rho)\right) |z_1-z_2|\\&=&4m C(N,q,\epsilon,\delta,A,\rho)|z_1-z_2|.
\earyst
This is the property (4). 
Obviously, $ \tilde{s}_{\beta,q,\epsilon,\delta,A}$ satisfies the other properties.
% so actually the Lemma \ref{lem:aa} is true for  $\beta=(\beta_0,\cdots,\beta_{N-1})\in  [-m,m]^N, \forall m>0, A^{(m)}_0(\epsilon,\delta,N)=A_0(\frac{\epsilon}{4m},\delta,N)$, In the following text we will use this general case, actually for $m=1$.
\end{proof}

\bigskip

\iffalse
Actually, we can construct it explicitly by the following :
\begin{lem}\label{lem:aa} \cite[Lemma 2.16]{BSK} Define $\tilde{s}_{\beta,q,\epsilon,A} $ by:\aryst
\tilde{s}_{\beta,q,\epsilon,A}(x)
&=& (\displaystyle{\sum_{j=0}^{N/2-1}}\beta_j (e^{-e^{-Asin2\pi(qx-j/N)}}-e^{-e^{-Asin2\pi(qx-(j+1)/N)}}))e^{-e^{-Asin2\pi qx}}\\
&+& (\displaystyle{\sum_{j=N/2}^{N-1}}\beta_j (e^{-e^{-Asin2\pi(qx-j/N)}}-e^{-e^{-Asin2\pi(qx-(j+1)/N)}}))e^{-e^{-Asin2\pi qx}}.
\earyst
 Then  $s_{\beta,q,\epsilon, A} $ extend holomorphically to $\C$ and $\frac{1}{q}$ periodic. Given $\epsilon,\delta >0$, for $A$ large enough, $\tilde{s}_{\beta,q,\epsilon,A} $ satisfying Lemma \ref{lem:aa}.
\end{lem}

Moreover, we have the following inequality:
\begin{lem}\label{lem:aa} \cite[Lemma 3.14]{BSK}  
$\forall \rho>0$ and $B_\rho=\{z \in \mathbb{C} \mid |\mathrm{Im}(z)|<\rho \}$, there exist a constant $C(\beta,q,\epsilon,A,\rho)$ such that:
$$\sup_{z_1,z_2\in B_\rho}|\tilde{s}_{\beta,q,\epsilon,A}(z_1)-\tilde{s}_{\beta,q,\epsilon,A}(z_2)|<C|z_1-z_2.|$$
\end{lem}

\begin{rema}\label{rem:bound}Note that $0<e^{-e^t}<1$ for all $t\in \R$. It is obviously that $|\tilde{s}_{\beta,q,\epsilon,A}(z)| < \displaystyle{\sum_{j=0}^{N-1}} |\beta_i|.$
\end{rema}
\fi
\section{A crucial lemma}
To prove the main theorem, we need the following lemma which is an analytic version of Lemma 6 in \cite{WZ}. 

Let $\Gamma = (\Q \times \R) \cup (\R \times \Q) \subset \R^2$. Recall that $\pi:\R^2\rightarrow \T^2$ is the covering projection. For $x, y \in \T^2$, we say $x-y \notin \Gamma$ means that if $\tilde{x},\tilde{y}\in \R^2$ satisfying $\pi(\tilde{x})=x, \pi(\tilde{y})=y$, then $\tilde{x}-\tilde{y} \notin \Gamma$, similarly, we say $x \notin \Gamma$ if $\tilde{x} \notin \Gamma$.
\begin{lem} \label{lem:h n}
Given an integer $q \geq 2$,  for any $\sigma > 0$ and $x, y \in \T^2$ with  $x,y,x-y \notin \Gamma$, there exists $(h,x',y',N)$
% $x',y',x'-y' \notin \Gamma$, 
 %$h \in \Diff_{\infty}^\omega(\T^2, \lambda)$ and even integer $N>3$ big enough
  such that:
\enmt
\item[$(1)$] $N\in 2\N\cap \N_{\geq4}$;
\item[$(2)$] $h \in \Diff_{\infty}^{\,\omega}(\T^2, \lambda)$ commutes with both $T_{(\frac{1}{q}, 0)}$ and $T_{(0, \frac{1}{q})}$;
\item[$(3)$]  $x',y',x'-y' \notin \Gamma, d(x, h(x')),d(y, h(y')), d(x', y') < \sigma$, and
\aryst
d(h T_{(\frac{2}{Nq}, 0)}(x'), h T_{(\frac{2}{Nq}, 0)}(y'))< \sigma;
\earyst
\item[$(4)$]  $d_{C^{0}}(\tilde{h}, \Id_{\R^2}) \leq 2d(x,y) + \frac{4}{Nq}$, where $\tilde{h}$ is a lift of $h$ to $\R^2$.
\eenmt
\end{lem}
\begin{proof}
%We identify $\T^2$ with $[0,1)\times[0,1)$ by the canonical projection. 
%we select $\tilde{x}_1,\tilde{y}_1\in [0,\frac{1}{q})\times [0,\frac{1}{q})$, such that $\tilde{x}-\tilde{x}_1 \in \frac{1}{q}\Z\times \frac{1}{q}\Z, \tilde{y}-\tilde{y}_1 \in \frac{1}{q}\Z\times \frac{1}{q}\Z.$
 %\footnote{It seems to me that (mod N) appeared here is not good, am I right?}, and $j_1+2< Nq$. 
Because $x, y, x-y \notin \Gamma$, we may select $\tilde{x},\tilde{y}\in [0,1)\times [0,1)$ such that $\pi(\tilde{x})=x, \pi(\tilde{y})=y$, $\tilde{x}-\tilde{y} \notin \Gamma$.  $\forall n \in \N_{\geq1}$, we may assume that
 $$\tilde{x}=(x_1,x_2)\in (\frac{i_1(n)}{nq},\frac{i_1(n)+1}{nq})\times(\frac{i_2(n)}{nq},\frac{i_2(n)+1}{nq}),$$
 $$\tilde{y}=(y_1,y_2)\in (\frac{j_1(n)}{nq},\frac{j_1(n)+1}{nq})\times(\frac{j_2(n)}{nq},\frac{j_2(n)+1}{nq}),$$
where $0\leq i_1(n), i_2(n), j_1(n), j_2(n)\leq nq-1$. When $n$ satisfies:
\begin{equation}\label{N}
\frac{1}{qn}< \min\left\{\min\limits_{k\in \Z, i=1,2}\left\{\left|x_i-y_i+\frac{k}{q}\right|\right\},\frac{1-y_1}{4}\right\},
\end{equation}
we have $i_t \not \equiv j_t \, (\mathrm{mod}\, n), t=1,2$. Now we fix an even 
integer $N>max\{3,\frac{4}{q\sigma}\}$  which satisfies (\ref{N}). Define $\alpha=(\alpha_0,\cdots,\alpha_{N-1}), \beta=(\beta_0,\cdots,\beta_{N-1})$ by:
\begin{equation}
\beta_i=\left\{
\begin{array}{rcl}
&\frac{j_1+1.5}{Nq}-x_1, & {i\equiv i_2(N) \,(\mathrm{mod}\,  N)};\\
&0 ,   & {\text{others}}
\end{array} \right.
\end{equation}

\begin{equation}
\alpha_i=\left\{
\begin{array}{rcl}
&\frac{j_2+0.5}{Nq}-x_2 ,& {i\equiv j_1(N)+1\, (\mathrm{mod} \, N)};\\
&0 ,  & {\text{others}}
\end{array} \right.
\end{equation}
We recall (\ref{eq:s}) in the last section. Let $\bar{s}_\alpha:=\bar{s}_{\alpha,q}, \bar{s}_\beta:=\bar{s}_{\beta,q}$.
We define the following maps defined on $[0,1)\times[0,1)$ to $\R^2$:
$$\bar{h}_\alpha(a,b)=(a,b-\bar{s}_\alpha(a)),\quad \bar{h}_\beta(a,b)=(a-\bar{s}_\beta(b),b),\quad\bar{h}=\bar{h}_\beta \bar{h}_\alpha.$$
% $\bar{s}_\alpha, \bar{s}_\beta$ can be extended periodically to the whole plane, so in the following we assume $\bar{s}_\alpha, \bar{s}_\beta$ is dedfined in$h$ is a 
Note that $\bar{h}^{-1}(a,b)=(a+\bar{s}_\beta(b),b+\bar{s}_\alpha(a+\bar{s}_\beta(b)))$. We have $\bar{h}^{-1}(\tilde{x})=(\frac{j_1+1.5}{Nq},\frac{j_2+0.5}{Nq})$ and $\bar{h}^{-1}(\tilde{y})=\tilde{y}$. Hence we get that
 $$d(\bar{h}^{-1}(\tilde{x}), \bar{h}^{-1}(\tilde{y}))<\frac{4}{Nq}<\sigma.$$
As $\frac{j_1+4}{Nq}<y_1+\frac{4}{Nq}<1$, we have $T_{(\frac{2}{Nq}, 0)}(\bar{h}^{-1}(\tilde{x})),T_{(\frac{2}{Nq}, 0)}(\bar{h}^{-1}(\tilde{y}))\in [0,1)^2$. By definition of $\bar{h}$, we obtain 
that $\bar{h}$ fixes $T_{(\frac{2}{Nq}, 0)}(\bar{h}^{-1}(\tilde{x}))$ and $T_{(\frac{2}{Nq}, 0)}(\bar{h}^{-1}(\tilde{y}))$. Therefore,
 $$d(\bar{h} T_{(\frac{2}{Nq}, 0)}(\bar{h}^{-1}(\tilde{x})), \bar{h} T_{(\frac{2}{Nq}, 0)}(\bar{h}^{-1}(\tilde{y})))=d(T_{(\frac{2}{Nq}, 0)}(\bar{h}^{-1}(\tilde{x})),T_{(\frac{2}{Nq}, 0)}(\bar{h}^{-1}(\tilde{y})))<\frac{4}{Nq}< \sigma.$$ 

Note that $\alpha,\beta \in [-1,1)^N$. Applying Lemma \ref{R}  for $m=1$, for any $\epsilon\in (0,\frac{1}{8}),\delta\in (0,1)$ and $A>A^{(1)}_0(\epsilon,\delta,N)$, we get the real analytic $1/q$-periodic approximate functions $\tilde{s}_\alpha:=\tilde{s}_{\alpha,q,\epsilon,\delta,A}, \tilde{s}_\beta:=\tilde{s}_{\beta,q,\epsilon,\delta,A}$ of $\bar{s}_\alpha,\bar{s}_\beta$, respectively. Define $\tilde{h}_\alpha=(a,b-\tilde{s}_\alpha(a)),\tilde{h}_\beta=(a-\tilde{s}_\beta(b),b)$ and let $\tilde{h}=\tilde{h}_\beta \tilde{h}_\alpha$. More precisely, $$\tilde{h}(a,b)=(a-\tilde{s}_\beta(b-\tilde{s}_\alpha(a)),b-\tilde{s}_\alpha(a)) \text{ for all } (a,b)\in\R^2.$$ Obviously, it is a diffeomorphism of $\R^2$. By Lemma \ref{R}, it can be extended holomorphically to $\C^2$. Moreover, it satisfies that $\tilde{h}((a,b)+(\frac{k_1}{q},\frac{k_2}{q}))=\tilde{h}(a,b)+(\frac{k_1}{q},\frac{k_2}{q})$ for any $(a,b)\in\C^2$ and $(k_1,k_2)\in\Z^2$. In particular, it induces a diffeomorphism on $\T^2$, denoted it by $h$. 
 
We note that $\tilde{x},\tilde{y},\bar{h}^{-1}(\tilde{x}),\bar{h}^{-1}(\tilde{y})\notin \Gamma_{Nq},$ where $\Gamma_{Nq}:=\{(a,b)\mid a=\frac{i}{Nq} ~ or ~ b=\frac{i}{Nq},i\in \Z \} $. We choose $\delta>0$ small enough such that $\tilde{x},\tilde{y},\bar{h}^{-1}(\tilde{x}),\bar{h}^{-1}(\tilde{y})$ not belong to the set $F_{q,N,\delta}$ which is defined in Lemma \ref{lem:aa}. By continuity and Lemma \ref{R}, letting $0<\epsilon<\min\left\{\left|\frac{j_2+0.5}{Nq}-x_2 \right|,\left|\frac{j_1+1.5}{Nq}-x_1\right|\right\}$ small enough, we have
% obtained by replacing $s_\alpha, s_\beta$ with $\bar{s}_\alpha, \bar{s}_\beta$, where $s_\alpha,s_\beta$ are the corresponding real analytic approximate functions of  $\bar{s}_\alpha, \bar{s}_\beta$ in Lemma \ref{lem:aa},  satisfies the conditions (1) and (2), and that
$$d(\tilde{h}^{-1}(\tilde{x}), \tilde{h}^{-1}(\tilde{y}))<\sigma,\quad d(\tilde{h} T_{(\frac{2}{Nq}, 0)}(\tilde{h}^{-1}(\tilde{x})),\tilde{h} T_{(\frac{2}{Nq}, 0)}(\tilde{h}^{-1}(\tilde{y})))<\sigma.$$
%\footnote{I think this formula is not useful. Is it better to write as $d(h T_{(\frac{2}{Nq}, 0)}(h^{-1}(x)),h T_{(\frac{2}{Nq}, 0)}(h^{-1}(y)))<\sigma$?}$$ 
Obviously, there exist $\tilde{x}',\tilde{y}'$ closed to $\tilde{h}^{-1}(\tilde{x}),\tilde{h}^{-1}(\tilde{y})$ and satisfy
$$d(\tilde{x}', \tilde{y}')<\sigma,\quad d(\tilde{h} T_{(\frac{2}{Nq}, 0)}(\tilde{x}'),\tilde{h} T_{(\frac{2}{Nq}, 0)}(\tilde{y}'))<\sigma.$$
 Assuming $x'=\pi(\tilde{x}'), y'=\pi(\tilde{y}')$, then $(h,x',y',N)$
 satisfies the items $(1)$  and $(3)$.
 %$ d(h\circ T_{(\frac{2}{Nq}, 0)}(x'), h\circ T_{(\frac{2}{Nq}, 0)}(y'))< \sigma4$

% we may extend $s_\alpha, s_\beta$ periodically to $\tilde{s}_\alpha, \tilde{s}_\beta$, then $\tilde{s}_\alpha, \tilde{s}_\beta$ is real analytic function defined on $\R$ ,

%$\tilde{h}$ is $\Z^2$ periodic, then it induce a differomorphism of $\T^2$, denoted it by $h$. Because $\tilde{h}$ preserves the measure 
Moreover,  as $\tilde{h}_\alpha=(a,b-\tilde{s}_\alpha(a)),\tilde{h}_\beta=(a-\tilde{s}_\beta(b),b)$ preserve the Lebesgue measure on $\R^2$,  $\tilde{h}$ preserves the Lebesgue measure. Therefore,  $h \in \Diff_{\infty}^{\,\omega}(\T^2, \lambda)$ and it commutes with $T_{(\frac{1}{q}, 0)}$ and $T_{(0, \frac{1}{q})}$. Furthermore, we have:
$$ d_{C^{0}}(\tilde{h}, \Id_{\R^2}) \leq \sup(|\tilde{s}_\alpha|+|\tilde{s}_\beta|) \leq \left|\frac{j_2+0.5}{Nq}-x_2 \right|+\left|\frac{j_1+1.5}{Nq}-x_1\right|<2d(x,y)+\frac{4}{Nq},$$
%\footnote{I think it is better to write the explicit form of $\tilde{h}$ before.}$$
where the second inequality comes from Remark \ref{rem:bound}. Then $(h,x',y',N)$  also satisfies $(2)$ and $(4)$, and hence it is the desired.
\end{proof}
\bigskip
 
\section{Proof of the main Theorem}

To prove Theorem \ref{thm example}, it is enough to prove the following proposition. 

\begin{prop}\label{prop example}
Fix $\rho>0$, there exists an area-preserving and minimal pseudo-rotation $f \in \Diff_\rho^{\,\omega}(\T^2,\lambda)$ which has bounded mean motion,  and satisfies the following: for any $\varepsilon > 0$, there exist two points $x, y \in \T^2$ with $d(x,y) < \varepsilon$, and an integer $N > 0$ such that $d(f^{N}(x), f^{N}(y)) \geq \frac{1}{1000}$.
\end{prop}

\begin{proof}
In the proof, we use the same approximation by conjugation scheme in \cite{WZ}. The different is that we have to replace the $C^\infty$ conjugacies in the proof of \cite[Proposition 6]{WZ} by the $C^\omega$ conjugacies in our situation.

We will construct a sequence of $h_n \in\Diff_\infty^{\,\omega}(\T^2, \lambda)$ , $\omega_n=(\omega_{n,1},\omega_{n,2}=q_n^{-1}\hat{\omega}_n) \in \Q^2$ with $\hat{\omega}_n\in \Z^2$, $q_n\in\N$ for each $n\geq1$.   We first introduce $(a1)_n-(a4)_n$ for a given $n \geq 1$:
\enmt
\item[$(a1)_{n}$] There exists $\tilde{h}_n$ a lift of $h_n$ to $\R^2$, such that $d_{C^0}(\tilde{h}_n, \Id_{\R^2}) < 2^{-n}$;
Let $H_n := h_{1} \cdots h_n$, then $H_n \in\Diff_\infty^{\,\omega}(\T^2, \lambda )$, and the map $\tilde{H}_n := \tilde{h}_1 \cdots \tilde{h}_n$ is a lift of $H_n$ and satisfies: $$d_{C^0}(\tilde{H}_n, \Id_{\R^2}) \leq \sum_{i=1}^{n}d_{C^0}(\tilde{h}_{i}, \Id_{\R^2}) < 1 - 2^{-n};$$
%\footnote{\,We stress that although the distances between $h_n$ and ${\rm Id_{\T^2}}$ are summable, $\{H_n\}$ does not converge. See $(a2)$. }
\item[$(a2)_{n}$] There exist $x_n, y_n \in \T^2$ with $x_n,y_n,x_n-y_n \notin \Gamma$ such that:  
                     $$d(x_n, y_n)  <  10^{-2n},\, d(H_n(x_n), H_n(y_n)) > \frac{1}{1000};$$
\item[$(a3)_{n}$] For $f_n := H_n T_{\omega_n} H_n^{-1}  \in\Diff_\infty^{\,\omega}(\T^2, \lambda )$, there exist $x^{(n)}, y^{(n)} \in \T^2$, $m_n \in \N$ such that: $$d(x^{(n)}, y^{(n)}) < 10^{-n},\, d(f^{m_n}_n(x^{(n)}), f^{m_n}_n(y^{(n)})) > \frac{1}{1000};$$
\item[$(a4)_{n}$] For any $z \in \T^2$, the set $\{f_n^{k}(z)\}_{k \in \Z}$ is $2^{-n}$-dense in $\T^2$.
\eenmt
Here we say that a set $K \subset \T^2$ is $\sigma$-dense for some $\sigma > 0$, if for any $x \in \T^2$ there exists $y \in K$ such that $d(x,y) < \sigma$.

Note that $(a1)_{n}$ and $(a3)_{n}$ imply the following:  the map $F_n := \tilde{H}_nT_{\omega_n} \tilde{H}_n^{-1}$ is a lift of $f_n$, and for  any integer $k \geq 1$ we have
\ary\label{eq:F}
\sup_{z \in \R^2}\norm{F_n^{k}(z) - z - k \omega_n} &=& \sup_{z \in \R^2}\norm{ \tilde{H}_n (\tilde{H}_n^{-1}(z) + k\omega_n) -\tilde{H}_n(\tilde{H}_n^{-1}( z)) - k \omega_n} \nonumber \\
&=&\sup_{z \in \R^2}\norm{ \tilde{H}_n (z + k\omega_n) -\tilde{H}_n(z) - k \omega_n}  \nonumber \\
&\leq& 2d_{C^0}(\tilde{H}_n, \Id_{\R^2})  < 10. \label{PartI3term 111}
\eary
Moreover, whenever $(a1)_n-(a4)_n$ are satisfied,  there exists a sufficiently small real number $\epsilon_n > 0$  such that 
for any $f \in \mathrm{Homeo}(\T^2)$ satisfying $d_{C^0}(f, f_n) < \epsilon_n$, for any $\omega\in \R^2$ satisfying $\norm{\omega - \omega_n} < \epsilon_n$, and for any $F \in \mathrm{Homeo}(\R^2)$ satisfying $d_{C^0}(F, F_{n}) < \epsilon_n$, we have
\ary 
&&\label{PartI3term 1} \norm{F^{k}(z) - z - k \omega} < 10, \quad \forall z \in \R^2, 1 \leq k \leq n, \\
&&  \label{PartI3term 2} d(f^{m_n}(x^{(n)}), f^{m_n}(y^{(n)})) > \frac{1}{1000}, \\
&& \label{PartI3term 4} \{f^{k}(z)\}_{k \in \N} \mbox{ is $2^{-n+1}$-dense in $\T^2$ for any } z \in \T^2.
\eary
%We can see that such $\epsilon_n$ exists by \eqref{PartI3term 111}, $(a3)_n$ and $(a4)_n$.
Without loss of generality, we can assume that $\epsilon_k > \epsilon_{k+1}$ for any  $k \geq 1$. 

Now we can introduce the last induction hypothesis for a given $n\geq1$:
%\enmt
\item[$(a5)_n$]\, we have $d_\rho(f_{n+1},f_{n}),  d_{C^0}(F_{n+1},F_n), \norm{\omega_{n+1} - \omega_{n}} < 2^{-n}\epsilon_{n}$.
% \footnote{I think you need define it before. And also you can define $d_\rho(f_{n+1},f_{n})$, which does not exist in \cite{BSK}? see also the last page.}
%\eenmt

For each integer $n\geq1$, we will construct $h_i,\omega_i,q_i,\hat{\omega}_i$ for $1\leq i\leq n$, satisfying $(a1)_i-(a4)_i$
for any $1\leq i\leq n$, and $(a5)_i$ for any $1\leq i\leq n-1$.

To start the induction, we let  $h_1 = \Id_{\T^2}$ and $\omega_{1} = q_1^{-1}\hat{\omega}_1 = (\frac{1}{100}, \frac{1}{10})$, where $q_1 = 100$ and $\hat{\omega}_1 = (1,10)$. It is direct to verify $(a1)_1 - (a4)_1$.
%Assume that we have constructed $(h_i, \omega_i, q_i, \hat{\omega}_i)$ for any $1 \leq i \leq n$, satisfying $(a1)_i - (a4)_i$ for any $1 \leq i \leq n$, and $(a5)_i$ for any $1 \leq i \leq n-1$. We construct $(h_{n+1}, \omega_{n+1}, q_{n+1}, \hat{\omega}_{n+1})$ as follows.

\textbf{Suppose that we have constructed $h_i,\omega_i,q_i,\hat{\omega}_i$ for $1\leq i\leq n$, satisfying $(a1)_i-(a4)_i$
for any $1\leq i\leq n$, and $(a5)_i$ for any $1\leq i\leq n-1$. Let $f_n, F_n, H_n$, $\omega_n$, $x_n$, $y_n$, $x^{(n)}$, $y^{(n)}$, $m_n$, $\epsilon_n$ be given by induction hypothesis. We will construct $(h_{n+1}, \omega_{n+1}, q_{n+1}, \hat{\omega}_{n+1})$ as follows.}

We recall that $\omega_n = (\omega_{n,1}, \omega_{n,2}) = q_n^{-1}\hat{\omega}_n$ with $\hat{\omega}_n  \in \Z^2$, $q_n \in \N$.  Without loss of generality, we can assume that $q_n > 10^{n}$. 

Recall $(a2)_n$, we set $\sigma_n$ small enough such that:
\enmt
\item [(1)] if $d(x, x_n),d(y, y_n)< \sigma_n$, then  $d(H_{n}(x), H_{n}(y))>\frac{1}{1000}$;
\item [(2)] $\sigma_n< 10^{-2n-2}\cdot \min\{\norm{DH_n}_{C^0}^{-1}, 1\}$, where $DH_n$ is the real derivative of $H_n$.
\eenmt

Applying Lemma \ref{lem:h n} to $q = q_n$, $x = x_n$, $y = y_n$ , $\sigma=\sigma_{n}$, we get $(h_{n+1}, x_{n+1}, y_{n+1},N_{n+1}) := (h, x', y',N)$.  Then 
\enmt
\item [(i)]$h_{n+1} \in \Diff_\infty^{\,\omega}(\T^2, \lambda)$ commutes with $T_{(\frac{1}{q_n}, 0)}$ and $T_{(0, \frac{1}{q_n})}$, so it also commutes with $T_{\omega_{n}}$,  and the lift of $h_{n+1}$, $\tilde{h}_{n+1}$ satisfies:
$$d_{C^0}(\tilde{h}_{n+1}, \Id_{\R^2}) \leq 2d(x_n, y_n) + \frac{4}{N_{n+1} q_n} < 2^{-n-1};$$

\item [(ii)] $d(x_{n+1}, y_{n+1}) < \sigma_n <10^{-2n-2};$

\item [(iii)]  $x_{n+1},y_{n+1},x_{n+1}-y_{n+1}\notin \Gamma,d(x_{n},  h_{n+1}(x_{n+1})),d(y_n,  h_{n+1}(y_{n+1}))< \sigma_n$, so by (1) above, we have
\begin{equation}\label{eq:>}
d(H_{n+1}(x_{n+1}), H_{n+1}(y_{n+1}))=d(H_{n}h_{n+1}(x_{n+1}), H_{n}h_{n+1}(y_{n+1}))>\frac{1}{1000}.\end{equation}
\eenmt
\textbf{This verifies $(a1)_{n+1},(a2)_{n+1}$.}

Let 
\aryst
z^{(n+1)}& :=& H_{n+1} T_{(\frac{2}{N_{n+1} q_n}, 0)}(z_{n+1}) \quad \mbox{for } z = x,y. \earyst
By Lemma \ref{lem:h n} (3) and the  choice of $\sigma_n$, we see that 
\aryst
d(x^{(n+1)}, y^{(n+1)}) \leq \norm{DH_n}_{C^0}d(h_{n+1}T_{(\frac{2}{N_{n+1}q_n}, 0)}(x_{n+1}), h_{n+1}T_{(\frac{2}{N_{n+1}q_n}, 0)}(y_{n+1})) < 10^{-n-1}.
\earyst
\textbf{This verify the first inequality in $(a3)_{n+1}$.} 

For any $\gamma \in \R^2$ , we set
\begin{equation}\label{eq:G}
G^{\gamma}_n := H_{n+1}  T_{(-\frac{2}{N_{n+1}q_n}, 0)+ \gamma} H_{n+1}^{-1}.
\end{equation}
By definition and (\ref{eq:>}), we have
\ary 
d(G^{(0,0)}_n(x^{(n+1)}), G^{(0,0)}_n(y^{(n+1)})) > \frac{1}{1000}.
\eary
Then by continuity, there exists $\kappa> 0$ such that for any $\gamma \in \R^2$ with $\norm{\gamma} < \kappa$, we have 
\ary \label{PartI3d G G}
d(G^{\gamma}_n(x^{(n+1)}), G^{\gamma}_n(y^{(n+1)})) > \frac{1}{1000}.
\eary
Without loss of generality, we can also assume that
\ary \label{PartI3kappa 2 n 1}
\kappa< 2^{-n-1} \norm{DH_{n+1}}^{-1}_{C^0}. \eary

Set $\omega_{n+1} = \omega_n + \eta_{n+1}$ for some $\eta_{n+1} \in \Q^{2} \setminus \{(0,0)\}$ of the form
\ary \label{eq:beta}
\eta_{n+1}= \frac{1}{q_{n}r_{n+1}}(1,v),
\eary
where $v, r_{n+1} \in \N$ satisfy that 
\ary \label{eq:v}
v > 100\kappa^{-1}, \, r_{n+1}\geq 100\kappa^{-1}v.
\eary

We write $\eta_{n+1}=(a_{n+1},b_{n+1})$ and select $\rho_n$ large enough such that $\tilde{H}^{-1}_n(B_p)\subset B_{\rho_n}$. 
%This is possible since  $$\norm{\tilde{H}^{-1}}_\rho=\sup \limits_{(z_1,z_2)\in B_\rho}|\tilde{H}^{-1}(z_1,z_2)|\leqslant \sup\limits_{\{(z_1,z_2) |  ~|\mathrm{Im}(z_1)|,|\mathrm{Im}(z_2)|\leqslant \rho, |\mathrm{Re}(z_1)|, |\mathrm{Re}(z_2)|\leqslant 1|\}}|\tilde{H}^{-1}(z_1,z_2)|.$$
By applying Lemma \ref{lem:h n} to $q = q_n$, $x = x_n$, $y = y_n$ and $\sigma=\sigma_{n}$, we recall that  $(\alpha, \beta, \epsilon,\delta,A)$ which appears in the proof Lemma \ref{lem:h n}. We can write down the explicit form of $h_{n+1}$ and $h_{n+1}^{-1}$ as:
 \aryst
\tilde{ h}_{n+1} T_{\pm\eta_{n+1}} \tilde{h}_{n+1}^{-1}(a,b)&=& ( a\pm a_{n+1}+\tilde{s}_\beta(b)-\tilde{s}_\beta(b\pm b_{n+1}+\tilde{s}_\alpha(a+\tilde{s}_\beta(b))-\tilde{s}_\alpha(a+\tilde{s}_\beta(b)\pm a_{n+1})),\\
 &&b\pm b_{n+1}+\tilde{s}_\alpha(a+\tilde{s}_\beta(b))-\tilde{s}_\alpha(a+\tilde{s}_\beta(b)\pm a_{n+1})).
\earyst

By  Lemma \ref{R}, there exists a constant $C_{n+1}:=C^{(1)}(N_{n+1}, q_n,\epsilon,\delta, A,\rho_n)$ such that:
 \begin{equation}\label{eq:C}
 \norm{\tilde{h}_{n+1}T_{\pm\eta_{n+1}} \tilde{h}_{n+1}^{-1}-\mathrm{Id}_{\mathbb{C}^2}}_{\rho_n}<|a_{n+1}|+C_{n+1}(C_{n+1}|a_{n+1}|+|b_{n+1}|)+|b_{n+1}|+C_{n+1} |a_{n+1}|.
 \end{equation}
We write $\rho_{n}':=|a_{n+1}|+C_{n+1}(C_{n+1} |a_{n+1}|+|b_{n+1}|)+|b_{n+1}|+C_{n+1} |a_{n+1}|$. To verify $(a5)_{n+1}$, we prove the following Lemma.
\begin{lem}\label{lem:drho} There is a positive number $Q$ large enough such that, when $r_{n+1}> Q$, we have
\enmt
\item[$(1)$] $\|\eta_{n+1}\|<\rho_{n}' < \left(\sup\limits_{z\in B_{\rho_n+1}}\norm{D\tilde{H}_{n}(z)}+1\right)^{-1}\cdot2^{-n}\epsilon_n$, where $\epsilon_n$ is determined by $(a1)_n-(a4)_n$ (see the paragraph below (\ref{eq:F})); 

\item[$(2)$] the map $f_{n+1}$ given by
 \aryst
f_{n+1} &:=& H_{n+1} T_{\omega_{n+1}} H_{n+1}^{-1} \\
&=& H_{n}(h_{n+1} T_{\omega_{n}} T_{\eta_{n+1}} h_{n+1}^{-1}) H_{n}^{-1} \\
&=& H_{n}(T_{\omega_n} h_{n+1} T_{\eta_{n+1}} h_{n+1}^{-1}) H_{n}^{-1}
\earyst
is $2^{-n}\epsilon_{n}$-close to $f_n$ in $\Diff_\rho^{\,\omega}(\T^2)$;

\item[$(3)$]   the map $F_{n+1} = \tilde{H}_{n+1}T_{\omega_{n+1}} \tilde{H}_{n+1}^{-1}$ is $2^{-n}\epsilon_{n}$-close to $F_n$ in $C^0(\R^2)$.
\eenmt
\end{lem}
\begin{proof}
By definitions, the item (1) is clear. We now show that (2) and (3) hold when (1) is satisfied.

Note that $\rho_n'<1$ when (1) is satisfied. By definition of $\tilde{d}_\rho$, we have
 \aryst
\tilde{d}_\rho(f_{n+1},f_n)&=&\norm{ \tilde{H}_{n}(T_{\omega_n} h_{n+1} T_{\eta_{n+1}} h_{n+1}^{-1}) \tilde{H}_{n}^{-1}-\tilde{H}_{n}T_{\omega_n} \tilde{H}_{n}^{-1}}_{\rho}\\
&\leq&\norm{ \tilde{H}_{n}T_{\omega_n} (h_{n+1} T_{\eta_{n+1}} h_{n+1}^{-1}) -\tilde{H}_{n}T_{\omega_n} }_{\rho_{n}}\\
&\leq&\left(\sup\limits_{z\in B_{\rho_n+\rho_n'}}\norm{D\tilde{H}_{n}(z)}\right)\cdot \norm{\tilde{h}_{n+1}T_{\eta_{n+1}} \tilde{h}_{n+1}^{-1}-\mathrm{Id}_{\mathbb{C}^2}}_{\rho_n}
\\
&\leq&2^{-n}\epsilon_n.
\earyst

Similarly, we have
  \aryst
\tilde{d}_\rho(f^{-1}_{n+1},f^{-1}_n)&=&\norm{ \tilde{H}_{n}(T_{-\omega_n} h_{n+1} T_{-\eta_{n+1}} h_{n+1}^{-1}) \tilde{H}_{n}^{-1}-\tilde{H}_{n}T_{-\omega_n} \tilde{H}_{n}^{-1}}_{\rho}\\
%&\leqslant&\left(\sup\limits_{z\in B_{\rho}}\norm{D\tilde{H}^{-1}_{n}(z)}\right)\cdot\norm{ \tilde{H}_{n}T_{\omega_n} (h_{n+1} T_{\eta_{n+1}} h_{n+1}^{-1}) -\tilde{H}_{n}T_{\omega_n} }_{\rho_{n}}\\
&\leq&\left(\sup\limits_{z\in B_{\rho_n+\rho_n'}}\norm{D\tilde{H}_{n}(z)}\right)\cdot \norm{\tilde{h}_{n+1}T_{-\eta_{n+1}} \tilde{h}_{n+1}^{-1}-\mathrm{Id}_{\mathbb{C}^2}}_{\rho_n}
\\
&\leq&2^{-n}\epsilon_n.
\earyst

Therefore,  by definition of $d_\rho$ and $\rho_n'$, the item (2) follows from (\ref{eq:C}) and the item (1). Finally, (3) is obvious from the proof of (2). \end{proof}
By Lemma \ref{lem:drho},  \textbf{we verify $(a5)_{n+1}$} by taking $r_{n+1}>Q$  and setting $q_{n+1} = q_n r_{n+1}$.

Note that for any $m = kq_{n}$ with $k \in \Z$, we have $m \omega_{n+1} = k\hat{\omega}_{n} + \frac{(k, kv)}{r_{n+1}}$, and hence
\aryst
f_{n+1}^{m} = H_{n+1} T_{ \frac{(k, kv)}{r_{n+1}}} H_{n+1}^{-1}.
\earyst
By  \eqref{PartI3kappa 2 n 1}, (\ref{eq:beta}) and (\ref{eq:v}), it is direct to see that:

(1) for any $\kappa$-dense subset of $\T^2$, denoted by $K$, the set $H_{n+1}(K)$ is $2^{-n-1}$-dense in $\T^2$;

(2) for any $z \in \T^2$,  $\{ (z + m\omega_{n+1}) \mod \Z^2 \}_{m \in q_n\N }$ is $\kappa$-dense in $\T^2$.

Thus for any $z \in \T^2$, the set $\{ f_{n+1}^{m}(z)  \}_{m \in \N} = \{H_{n+1}(H_{n+1}^{-1}(z) + m \omega_{n+1})   \}_{m \in \N}$ is $2^{-n-1}$-dense in $\T^2$. \textbf{This verifies $(a4)_{n+1}$}. Moreover, by (\ref{eq:G}) and the item (2) above, there exists some $m \in q_n\N$ such that $f_{n+1}^{m} = G_{n}^{\gamma}$
for certain $\gamma\in\R^2$ with $\norm{\gamma} < \kappa$. 
Then by \eqref{PartI3d G G}, \textbf{we verify the second inequality in  $(a3)_{n+1}$}.

The above discussions show that, by choosing $r_{n+1}$ sufficiently large, we can ensure that  $(h_{n+1}$, $\omega_{n+1}, q_{n+1}, \hat{\omega}_{n+1})$  satisfies $(a1)_{n+1}-(a4)_{n+1}$ and $(a5)_n$, and thus complete the induction. For each $1\leq i\leq5$, let us denote by $(ai)$ the collection of induction hypotheses $(ai)_1,(ai)_2,\ldots$

Let the sequence $\{ f_n \}_{n \geq 1}$ be constructed by the above induction scheme.
By $(a5)$, $\{ f_n \}_{n \geq 1}$ converges to some map $f \in \Diff_\rho^{\,\omega}(\T^2, \lambda)$ under the $d_\rho$-metric; $\{ \omega_n \}_{n \geq 1}$ converges to some $\omega \in \R^2$; and $\{F_n\}_{n\geq1}$ converges to some $F\in\mathrm{Homeo}(\R^2)$, which is clearly a lift of $f$. Moreover for any integer $n \geq 1$, we have $ d_\rho(f_{n}, f), d_{C^0}(F_{n}, F), \norm{\omega - \omega_n} < \epsilon_n$. Consequently, (\ref{PartI3term 1}) to (\ref{PartI3term 4}) holds for $(f,\omega,F)$ and every $n\geq1$. By \eqref{PartI3term 2} and (a3), for any $\epsilon > 0$,  there exist $x,y \in \T^2$ satisfying $d(x,y) < \epsilon$, and an integer $m > 0$, such that $d(f^{m}(x), f^{m}(y)) \geq \frac{1}{1000}$.    By \eqref{PartI3term 4} and \eqref{PartI3term 1}, $f$ is  minimal and of bounded mean motion, thus $f$ must be a totally irrational pseudo-rotation. This concludes the proof. 
\end{proof}
By Theorem \ref{thm:J},  the map $f$ we constructed above is a real-analytic,  homotopic to the identity and semi-conjugate to a minimal translation $T_\a$. It is not conjugate to any minimal translation because for any $\varepsilon > 0$, there exist two points $x, y \in \T^2$ with $d(x,y) < \varepsilon$, and an integer $N > 0$ such that $d(f^{N}(x), f^{N}(y)) \geq \frac{1}{1000}$. Finally, Theorem \ref{thm example} follows from a similar argument in the proof of \cite[Theorem 4]{WZ}. We omit it.

\begin{rema} It is not difficult to check that the rotation vector of the constructed map $f$ above satisfies the  super-Liouvillian condition (see \cite[Theorem 2]{WZ} for the definition) by our choice of $q_n$ (see Remark \ref{rem:e} and Lemma \ref{lem:drho}). Based on  \cite[Theorem 2]{WZ}, the map $f$  is  $C^\infty$-rigid. 
\end{rema}

\end{document}